\documentclass[11pt]{amsart}

\pdfoutput=1

\usepackage{amsthm,amsmath,amssymb,amsbsy,amsfonts,enumitem,graphicx}

\usepackage[hmargin=1.5in,vmargin=1.5in]{geometry}

\usepackage{hyperref}

\hypersetup{
    colorlinks=true,
    linkcolor=black,
    citecolor=black,
    filecolor=black,
    urlcolor=black,
}

\newtheorem{theorem}{Theorem}[section]
\newtheorem{lemma}[theorem]{Lemma}
\newtheorem{corollary}[theorem]{Corollary}

\newtheorem{conjecture}[theorem]{Conjecture}

\theoremstyle{definition}

\newtheorem{remark}[theorem]{Remark}

\def\N{\mathbb{N}}
\def\C{\mathbb{C}}
\def\R{\mathbb{R}}
\def\Z{\mathbb{Z}}

\DeclareMathOperator{\dist}{dist}
\DeclareMathOperator{\maxdist}{maxdist}
\DeclareMathOperator{\erfc}{erfc}

\renewcommand*{\Re}{\operatorname{Re}}
\renewcommand*{\Im}{\operatorname{Im}}

\begin{document}

\title[Zeros of sections of exponential integrals]{Limit curves for zeros of sections of\\ exponential integrals}
\author{Antonio R. Vargas}
\address{Dept. of Mathematics and Statistics, Dalhousie University, Halifax, Nova Scotia B3H 4J5, Canada}
\email{antoniov@mathstat.dal.ca}

\begin{abstract}
We are interested in studying the asymptotic behavior of the zeros of partial sums of power series for a family of entire functions defined by exponential integrals.  The zeros grow on the order of $O(n)$, and after rescaling we explicitly calculate their limit curve.  We find that the rate that the zeros approach the curve depends on the order of the singularities/zeros of the integrand in the exponential integrals.  As an application of our findings we derive results concerning the zeros of partial sums of power series for Bessel functions of the first kind.


\end{abstract}

\maketitle

\ifpdf
    \graphicspath{{sec_intro/PNG/}{sec_intro/PDF/}}
\else
    \graphicspath{sec_intro/EPS/}
\fi

\section{Introduction}
\label{sec_intro}

In this paper we are concerned with the asymptotic behavior of the zeros of the polynomial sequence given by the partial sums of a convergent power series.  If $f$ is a function which is analytic at the origin, then it can be represented by a power series
\[
	f(z) = \sum_{k=0}^{\infty} a_k z^k
\]
which converges near $z=0$.  We denote the $n^{\text{th}}$ partial sum of this power series by
\begin{equation}
\label{sndef}
	s_n[f](z) = \sum_{k=0}^{n} a_k z^k,
\end{equation}
and we refer to $s_n[f](z)$ as the $n^{\text{th}}$ section of $f$.

If the power series for $f$ has a finite radius of convergence, then it is a classical result of Jentzsch \cite{jentzsch} that every point on the circle of convergence of the power series will be a limit point of the zeros of the sections $s_n[f](z)$.  In addition, because power series converge uniformly on compact subsets of their domains of convergence, Hurwitz's theorem (see, e.g., \cite[p.~4]{marden:geom}) tells us that any zero of $f$ inside the radius of convergence will also be a limit point of the zeros of the sections.

The behavior of the zeros becomes much more interesting when $f$ is entire.  The topic can be traced back to Szeg\H{o}, who studied in \cite{szego:exp} the sections of the exponential function $e^z$, given by
\[
	s_n[\exp](z) = \sum_{k=0}^{n} \frac{z^k}{k!}.
\]
Szeg\H{o} found that the zeros of the normalized sections $s_n[\exp](nz)$ have as their set of limit points the simple closed loop
\[
	D = \left\{z \in \C : |z| \leq 1 \,\,\,\text{and}\,\,\, \left|z e^{1-z}\right| = 1\right\}.
\]
This is often referred to as the Szeg\H{o} curve.  The rate that the zeros approach this curve was first studied by Buckholtz \cite{buckholtz:expcharacter}, who showed that every zero of $s_n[\exp](nz)$ lies within a distance of $2e/\sqrt{n}$ of $D$.  Carpenter, Varga, and Waldvogel \cite{cvw:expasympi} examined this phenomenon in detail and showed that Buckholtz's result gives the best-possible asymptotic order.  The statement of the theorem involves the complementary error function
\[
	\erfc{z} = \frac{2}{\sqrt{\pi}} \int_{z}^{\infty} e^{-t^2}\,dt,
\]
where the path of integration begins at $z$ and travels to the right to $\infty$.

\begin{theorem}[CVW]
\label{cvwtheorem1}
For $\Omega \subseteq \C$, define $\maxdist\!\left(\Omega,\, C\right) = \sup_{z \in \Omega} \left\{\dist(z,C)\right\}$.  If $\{z_{k,n}\}_{k=1}^{n}$ are the zeros of $s_n[\exp](nz)$ and if $t_1$ is the zero of the complementary error function $\erfc$ closest to the origin in the upper half-plane, then
\[
	\liminf_{n \to \infty} \sqrt{n} \cdot \maxdist\!\left(\{z_{k,n}\}_{k=1}^{n},D\right) \geq \Re(t_1) + \Im(t_1) \approx 0.636657.
\]
\end{theorem}

The authors also showed that the zeros which are bounded away from the point $z=1$ approach the Szeg\H{o} curve more quickly.

\begin{theorem}[CVW]
\label{cvwtheorem2}
Let $C_\delta$ be the ball of radius $\delta$ centered at $z=1$.  If $\{z_{k,n}\}_{k=1}^{n}$ are the zeros of $s_n[\exp](nz)$ and if $\delta$ is any fixed number with $0 < \delta \leq 1$, then
\[
	\maxdist\!\left(\{z_{k,n}\}_{k=1}^{n} \setminus C_\delta,\, D\right) = O\!\left(\frac{\log n}{n}\right)
\]
as $n \to \infty$.
\end{theorem}

There have been some similar results for other power series, notably those in a memoir by Edrei, Saff, and Varga concerning the Mittag-Leffler functions \cite{esv:sections} which was published prior to the CVW paper mentioned above.  Other examples include the sine and cosine (see \cite{vc:sincosasympiii} and the references therein), a class of confluent hypergeometric functions \cite{norfolk:1f1}, finite sums of exponentials \cite{mallison:expsums}, and even some classes of divergent power series \cite{dilcher:divergent}.  For a survey of this topic the reader is referred to \cite{mymscthesis}.

In this paper we will continue in this vein, studying power series defined by exponential integrals of a certain form, which we introduce in \hyperref[sec_prelims]{Section \ref*{sec_prelims}}.  Special cases of these functions include the aforementioned confluent hypergeometric functions, the prolate spheroidal wave functions (see \cite[sec.~V]{prolatespheroidal}), and Bessel functions of the first kind, the last of which we will discuss in detail in \hyperref[sec_specialcases]{Section \ref*{sec_specialcases}}.

Our results were particularly inspired by the work of Norfolk \cite{norfolk:1f1} on the confluent hypergeometric functions, defined by
\[
	{}_1F_1(1;b;z) = \Gamma(b) \sum_{k=0}^{\infty} \frac{z^k}{\Gamma(k+b)}.
\]
Norfolk studied the case where $b$ is real and $b \neq 1, 0, -1, -2, \ldots$.  His main tool was an exponential integral representation of the functions valid for $b > 1$ (see \hyperref[sec_specialcases]{Section \ref*{sec_specialcases}}).

Norfolk also studied a related family of integral transforms in \cite{norfolk:transforms}.

The results in this paper were originally obtained in the author's Master's thesis \cite{mymscthesis} and form an analogue to \hyperref[cvwtheorem2]{Theorem \ref*{cvwtheorem2}}.


\ifpdf
    \graphicspath{{sec_prelims/PNG/}{sec_prelims/PDF/}}
\else
    \graphicspath{sec_prelims/EPS/}
\fi

\section{Definitions and preliminaries}
\label{sec_prelims}

The functions we are interested in are defined by integrals of the form \linebreak $\int_{-a}^{b} \varphi(t) e^{zt}\,dt$.  The restrictions we place on the function $\varphi$ are determined essentially by the abilities of Watson's lemma, discussed below.

Suppose $0 \leq a,b < \infty$ and let $\varphi : [-a,b] \to \C \cup \{\infty\}$ be a measurable function satisfying
\[
	\int_{-a}^{b} |\varphi(t)|\,dt < \infty
\]
and $\varphi(t) = (t+a)^{\mu} f_1(t+a) = (b-t)^{\nu} f_2(b-t)$, where
\begin{enumerate}[label=(\arabic*)]
\item $\mu,\nu \in \C$ with $\Re(\mu) > -1$ and $\Re(\nu) > -1$,
\item $f_1,f_2 : [0,a+b] \to \C \cup \{\infty\}$ with $f_1(0)$ and $f_2(0)$ both finite and nonzero,
\item in a neighborhood of $t=0$, both $f_1'(t)$ and $f_2'(t)$ exist and are bounded.
\end{enumerate}
Define
\[
	F(z) = \int_{-a}^{b} \varphi(t) e^{zt}\,dt.
\]
It is a consequence of the dominated convergence theorem that the function $F$ is entire, and its sections are given by the formula
\begin{equation}
\label{snFdef}
	s_n[F](z) = \sum_{k=0}^{n} \frac{z^k}{k!} \int_{-a}^{b} \varphi(t) t^k\,dt.
\end{equation}
We can view $F(z)$ as the exponential generating function of these particular integral moments of $\varphi$.

Properties $(1)$, $(2)$, and $(3)$ above describe how the function $\varphi$ behaves near the endpoints of integration, and perhaps more importantly they name various quantities we will refer to throughout the paper.  Property $(3)$ serves a special purpose: it allows us to determine a simple error term in the asymptotic expansion of $F$ through the use of Watson's lemma (\hyperref[watsonlemma]{Theorem \ref*{watsonlemma}}).

Next we collect some theorems which will aid us in proving our results.  The first such theorem is due to Rosenbloom \cite{rosen:thesis} (see also \cite{rosen:distrib} for a summary of this reference).

A sequence of sections $\{s_N[f](z)\}$ is said to have a positive fraction of zeros in any sector with vertex at the origin if
\[
	\liminf_{N \to \infty} \frac{\#_N^{\angle}(\theta_1,\theta_2)}{N} > 0
\]
for any fixed $\theta_1$ and $\theta_2$, where $\#_N^{\angle}(\theta_1,\theta_2)$ is the number of zeros of the section $s_N[f](z)$ in the sector $\theta_1 \leq \arg z \leq \theta_2$.  Rosenbloom's result is as follows.

\begin{theorem}[Rosenbloom]
\label{rosentheo}
Let
\[
	f(z) = \sum_{k=0}^{\infty} a_k z^k
\]
be an entire function of finite positive order $\rho$ and let $s_n[f](z)$ be as in equation \eqref{sndef}.  Define $\rho_n = |a_n|^{-1/n}$.  There is an increasing sequence of indices $\{N\}$ such that the sequence of sections $\{s_N[f](z)\}$ has a positive fraction of zeros in any sector with vertex at the origin and, for every $\epsilon > 0$, the number of zeros of $s_N[f](z)$ satisfying
\[
	|z| \geq \left(e^{1/\rho} + \epsilon\right)\rho_N
\]
is bounded and all zeros lie in the disk
\[
	|z| \leq \left(2e^{1/\rho} + \epsilon\right)\rho_N
\]
for $N$ large enough.
\end{theorem}

The means by which such a sequence of indices $\{N\}$ can be constructed is given by Norfolk in \cite{norfolk:widthconj}.  In doing so, Norfolk furnishes a constructive proof of the above result.

We will also require Watson's lemma on the asymptotic behavior of exponential integrals.  We refer the reader to \cite{miller:aaa} for a thorough discussion of this result.  In the following, $\lambda$ is a complex parameter.

\begin{theorem}[Watson's lemma]
\label{watsonlemma}
Suppose $0 < T \leq \infty$ and $\varphi : [0,T] \to \C \cup \{\infty\}$ is a function satisfying
\[
	\int_0^T |\varphi(t)|\,dt < \infty
\]
and $\varphi(t) = t^{\sigma}h(t)$, where $\Re(\sigma) > -1$, $h(0) \neq 0$, and $h'(t)$ exists and is bounded in a neighborhood of $t=0$.  Then the exponential integral
\[
	\Phi(\lambda) = \int_0^T \varphi(t) e^{-\lambda t}\,dt
\]
is finite for all $\Re(\lambda) > 0$, and
\[
	\Phi(\lambda) = \frac{h(0) \Gamma(\sigma+1)}{\lambda^{\sigma+1}} + O\!\left(\lambda^{-\sigma-2}\right)
\]
as $\lambda \to \infty$ with $|\arg \lambda| \leq \theta$ for any fixed $0 \leq \theta < \pi/2$.
\end{theorem}

Though this form of Watson's lemma only gives an asymptotic for $\Phi(\lambda)$ as $\lambda \to \infty$ to the right, it can easily be extended to address the case when $\lambda \to \infty$ to the left if we assume that $T$ is finite.

\begin{corollary}
\label{watsoncoro}
Suppose $0 < T < \infty$ and $\varphi : [0,T] \to \C \cup \{\infty\}$ is a function satisfying
\[
	\int_0^T |\varphi(t)|\,dt < \infty
\]
and $\varphi(t) = (T-t)^{\sigma}h(T-t)$, where $\Re(\sigma) > -1$, $h(0) \neq 0$, and $h'(t)$ exists and is bounded in a neighborhood of $t=0$.  Then the exponential integral
\[
	\Phi(\lambda) = \int_0^T \varphi(t) e^{\lambda t}\,dt
\]
is finite for all $\Re(\lambda) > 0$, and
\[
	\Phi(\lambda) = \frac{h(0) \Gamma(\sigma+1)}{\lambda^{\sigma+1}} \,e^{T\lambda} + O\!\left(\lambda^{-\sigma-2} e^{T\lambda}\right)
\]
as $\lambda \to \infty$ with $|\arg \lambda| \leq \theta$ for any fixed $0 \leq \theta < \pi/2$.
\end{corollary}

\begin{proof}
We have
\begin{align*}
\Phi(\lambda) e^{-T\lambda} &= \int_0^T \varphi(t) e^{-\lambda(T-t)}\,dt \\
			&= \int_0^T \varphi(T-u) e^{-\lambda u}\,du \\
			&= \int_0^T u^{\sigma} h(u) e^{-\lambda u}\,du \\
			&= \frac{h(0) \Gamma(\sigma+1)}{\lambda^{\sigma+1}} + O\!\left(\lambda^{-\sigma-2}\right)
\end{align*}
as $\lambda \to \infty$ with $|\arg \lambda| \leq \theta$ for any fixed $0 \leq \theta < \pi/2$, by Watson's lemma.
\end{proof}

\ifpdf
    \graphicspath{{sec_curveasymp/PNG/}{sec_curveasymp/PDF/}}
\else
    \graphicspath{sec_curveasymp/EPS/}
\fi

\section{Main results}
\label{sec_curveasymp}

Recall from \hyperref[sec_prelims]{Section \ref*{sec_prelims}} that we are concerned with functions of the form
\[
	F(z) = \int_{-a}^{b} \varphi(t) e^{zt} \,dt
\]
with $\varphi$ satisfying some light requirements.  The $n^\text{th}$ section of $F$ is the polynomial
\[
	s_n[F](z) = \sum_{k=0}^{n} \frac{z^k}{k!} \int_{-a}^{b} \varphi(t) t^k \,dt.
\]

The statements and proofs of the main results of this section depend on the relative sizes of $a$ and $b$ and of $\Re(\mu)$ and $\Re(\nu)$.  To this end, define
\[
	c = \max\{a,b\}
\]
and
\[
	\xi = \begin{cases}
			\Re(\mu) & \text{if } a > b, \\
			\Re(\nu) & \text{if } a < b, \\
			\min\{\Re(\mu),\Re(\nu)\} & \text{if } a = b.
		  \end{cases}
\]

Further, let
\begin{equation}
	U = \left\{z \in \C : \left|ze^{1-z}\right| > 1 \right\} \cup \left\{z \in \C : \left|ze^{1-z}\right| \leq 1 \,\ \text{and} \,\ \Re(z) < 1 \right\}
\label{uregiondef}
\end{equation}
be the open region shown in \hyperref[uregion]{Figure \ref*{uregion}}.  Define
\begin{equation}
	V_{a,b} = \left\{z \in \C : -az \in U \,\ \text{and} \,\ bz \in U\right\}.
\label{vregiondef}
\end{equation}

Lastly, in the theorems below we require that, if $a=b$ and $\Re(\mu) = \Re(\nu)$, the indices $\{n\}$ of the sections are chosen so that quantity
\begin{equation}
\label{ncond}
	(-1)^n f_1(0) \Gamma(\mu+1) + f_2(0) \Gamma(\nu+1) a^{\nu-\mu} n^{\mu-\nu}
\end{equation}
is bounded away from $0$.  This condition is imposed to ensure that we can use the asymptotic representations derived in \hyperref[coeffasymp]{Lemma \ref*{coeffasymp}} and \hyperref[tailasymp]{Lemma \ref*{tailasymp}} without incident.

\begin{figure}[h!tb]
	\centering
	\includegraphics[width=0.4\textwidth]{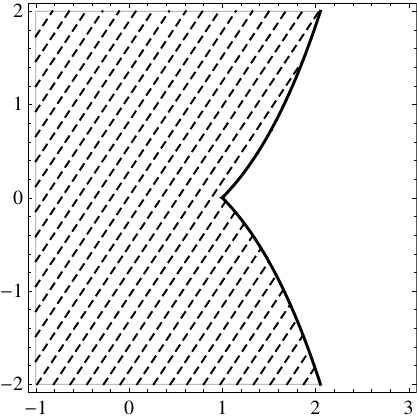}
	\caption{Dashed lines indicate the open region $U$ defined in \eqref{uregiondef}.  The curve $\left|ze^{1-z}\right| = 1$, $\Re(z) \geq 1$, shown as a solid line, is the boundary of this region.}
\label{uregion}
\end{figure}

\begin{theorem}
\label{curvetheo1}
Let $F$ be an exponential integral function as described in \hyperref[sec_prelims]{Section \ref*{sec_prelims}} and let $\{n\}$ be any subsequence of the natural numbers.
\begin{enumerate}[label=(\roman*)]
\item Let $\{z_n\}$ be a sequence of complex numbers such that $s_n[F](nz_n) = 0$ for all $n$ which has a limit point in the region $V_{a,b} \cap \{z \in \C : \Re(z) < 0\}$.  Then the elements of the sequence satisfy
\[
	\left|cz_ne^{1+az_n}\right| = 1 + \left(\xi - \Re(\mu) + \frac{1}{2}\right) \frac{\log n}{n} + O(1/n)
\]
as $n \to \infty$.
\item Let $\{z_n\}$ be a sequence of complex numbers such that $s_n[F](nz_n) = 0$ for all $n$ which has a limit point in the region $V_{a,b} \cap \{z \in \C : \Re(z) > 0\}$.  Then the elements of the sequence satisfy
\[
	\left|cz_ne^{1-bz_n}\right| = 1 + \left(\xi - \Re(\nu) + \frac{1}{2}\right) \frac{\log n}{n} + O(1/n)
\]
as $b \to \infty$.
\end{enumerate}
\end{theorem}

\begin{theorem}
\label{curvetheo2}
Let $\{N\}$ be a sequence of the natural numbers as described in \hyperref[rosentheo]{Theorem \ref*{rosentheo}}.  It is true that
\begin{enumerate}[label=(\roman*)]
\item Every point on the curve
\begin{align*}
	D_{a,b} &= \left\{z \in \C : \Re(z) \leq 0,\,\,\, |z| \leq \frac{1}{c},\,\,\, \text{and } \left|cze^{1+az}\right| = 1\right\} \\
		&\qquad \cup \left\{z \in \C : \Re(z) \geq 0,\,\,\, |z| \leq \frac{1}{c},\,\,\, \text{and } \left|cze^{1-bz}\right| = 1\right\}
\end{align*}
is a limit point of the zeros of the normalized sections $s_N[F](Nz)$.
\item The limit points of the zeros of the normalized sections $s_N[F](Nz)$ on the imaginary axis lie on the line segment
\[
	D_{\textnormal{imag}} = \left\{z \in \C : \Re(z) = 0 \,\,\,\text{and}\,\,\, |z| \leq \frac{1}{ec} \right\}.
\]
\item Every limit point of the zeros of the normalized sections $s_N[F](Nz)$ not in $D_{a,b} \cup D_\textnormal{imag} \cup\{\pm 1/c\}$ is isolated and lies in the closed region
\[
	E_{a,b} = \left\{z \in \C : |z| \leq \frac{2}{c}\right\} \setminus V_{a,b}.
\]
\end{enumerate}
\end{theorem}

\begin{remark}
\label{curvetheoremark}
The expressions in parts $(i)$ and $(ii)$ of \hyperref[curvetheo1]{Theorem \ref*{curvetheo1}} give information about whether the zeros eventually lie on the inside or the outside of the limit curve based on the signs of the quantities $\xi - \Re(\mu) + 1/2$ and $\xi - \Re(\nu) + 1/2$.  For example, if $\xi - \Re(\mu) + 1/2 > 0$, then it will eventually be true that the zeros which approach $D_{a,b}$ in the left half-plane will satisfy $\left|cze^{1+az}\right| > 1$ and hence will lie outside of $D_{a,b}$.  If either of the quantities $\xi - \Re(\mu) + 1/2$ or $\xi - \Re(\nu) + 1/2$ is zero then the theorem does not give any information about the direction from which the zeros approach the relevant part of the curve.
\end{remark}

The curve $D_{a,b}$ contains at least one of the points $\pm 1/c$, so we can interpret part (iii) of \hyperref[curvetheo2]{Theorem \ref*{curvetheo2}} to mean that at most one of any limit points not on $D_{a,b} \cup D_\textnormal{imag}$ is not isolated.  This limit point, if it exists, is the element of the pair $\pm 1/c$ which is not on $D_{a,b}$.

\hyperref[eiplot1]{Figure \ref*{eiplot1}} and \hyperref[eiplot2]{Figure \ref*{eiplot2}} illustrate two interesting cases of the results in \hyperref[curvetheo1]{Theorem \ref*{curvetheo1}} and \hyperref[curvetheo2]{Theorem \ref*{curvetheo2}}.  The reader may refer to \cite{mymscthesis} for more plots of this kind.  Two examples of the set $E_{a,b}$ defined in part (iii) of \hyperref[curvetheo2]{Theorem \ref*{curvetheo2}} are shown in \hyperref[eiplot1]{Figure \ref*{extrazeros}}.

\begin{figure}[htb]
	\centering
	\includegraphics[width=0.9\textwidth]{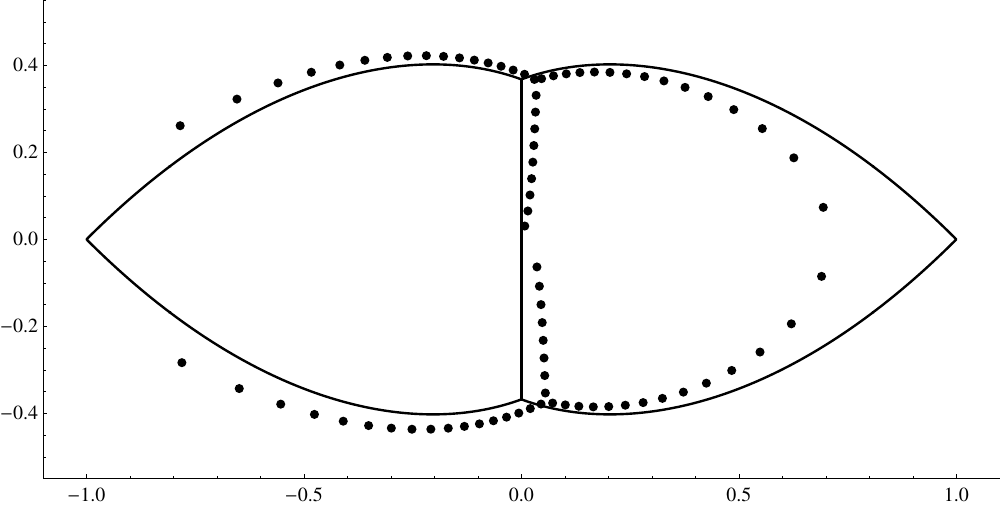}
	\caption{Zeros of the normalized section $s_{80}[F](80z)$ and the limit curve $D_{a,b}\cup D_{\textnormal{imag}}$ with $a=b=1$ and $\varphi(t) = (1-t)^{3/2}(1+t)^{-1/2+i}$.  Note that $\xi - \Re(\nu) + 1/2 = -3/2 < 0$, which predicts that the zeros in the right half-plane will approach the limit curve from the interior.}
\label{eiplot1}
\end{figure}

\begin{figure}[t!]
	\centering
	\includegraphics[width=0.9\textwidth]{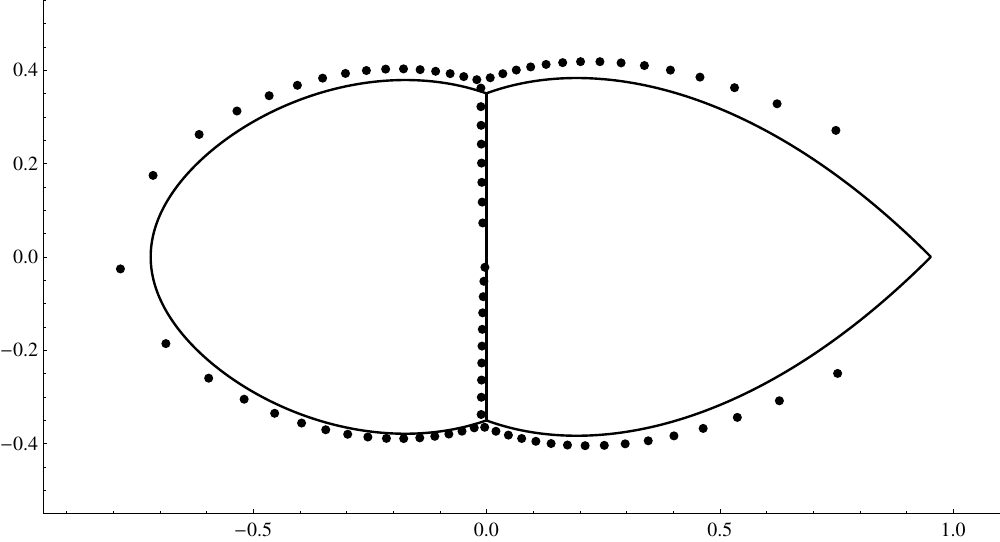}
	\caption{Zeros of the normalized section $s_{80}[F](80z)$ and the limit curve $D_{a,b}\cup D_{\textnormal{imag}}$ with $a=1$, $b=\frac{21}{20}$, and $\varphi(t) = (\frac{21}{20}-t)^{i}(1+t)^{1/2-i}$.  Note that $\xi - \Re(\mu) + 1/2 = 0$, which predicts that the zeros in the left half-plane will approach the limit curve at a rate of $O(1/n)$.}
\label{eiplot2}
\end{figure}

To prove these theorems we will require a few lemmas, the first of which concerns the asymptotic behavior of the function $F$.

\begin{lemma}
\label{Fasymp}
As $n \to \infty$,
\[
	F(nz) = f_1(0) \Gamma(\mu+1) (-nz)^{-\mu-1} e^{-anz} \Bigl(1 + O(1/n) \Bigr)
\]
when $z$ is restricted to a compact subset of $\Re(z) < 0$, and
\[
	F(nz) = f_2(0) \Gamma(\nu+1) (nz)^{-\nu-1} e^{bnz} \Bigl(1 + O(1/n) \Bigr)
\]
when $z$ is restricted to a compact subset of $\Re(z) > 0$.
\end{lemma}

\begin{proof}
This follows from a direct application of \hyperref[watsoncoro]{Corollary \ref*{watsoncoro}}.  To see this, suppose first that $z$ is restricted to a compact subset of $\Re(z) < 0$, and make the substitution $t = b - s$ in the integral for $F(nz)$ to get
\[
	F(nz) = \int_{-a}^{b} \varphi(t) e^{nzt}\,dt = e^{bnz} \int_{0}^{a+b} \varphi(b-s) e^{-nzs}\,ds,
\]
which, after replacing $z$ with $-z$, is of the form required by the corollary.  Next, suppose that $z$ is restricted to a compact subset of $\Re(z) > 0$, and make the substitution $t = s - a$ in the definition of $F(nz)$ to get
\[
	F(nz) = \int_{-a}^{b} \varphi(t) e^{nzt}\,dt = e^{-anz} \int_{0}^{a+b} \varphi(s-a) e^{nzs}\,ds,
\]
which is also of the required form.
\end{proof}

\begin{figure}[t!]
	\centering
	\begin{tabular}{cc}
		\includegraphics[width=0.47\textwidth]{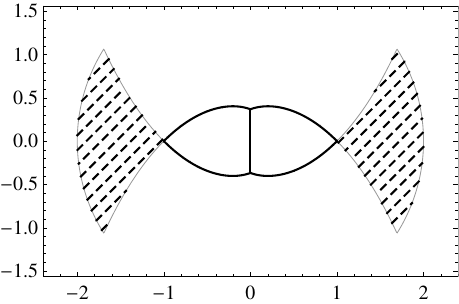}
			& \includegraphics[width=0.47\textwidth]{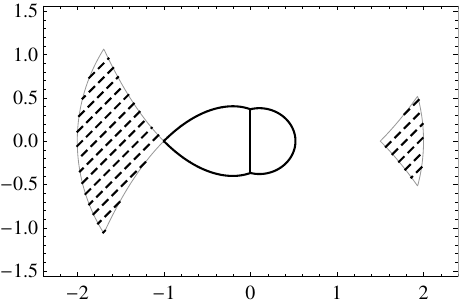}
	\end{tabular}
	\caption{Examples of the regions where additional discrete limit points may lie.  \textbf{Left:} The closed region $E_{1,1}$ (defined in part (iii) of \hyperref[curvetheo2]{Theorem \ref*{curvetheo2}}) is indicated by dashed lines.  The corresponding limit curve $D_{1,1} \cup D_\textnormal{imag}$ is drawn with solid lines.  \textbf{Right:} The closed region $E_{1,2/3}$ is indicated by dashed lines and the curve $D_{1,2/3}\cup D_\textnormal{imag}$ is drawn with solid lines.}
	\label{extrazeros}
\end{figure}

We must also find asymptotics for the integral moments of $\varphi$ and hence for the power series coefficients of $F$.

\begin{lemma}
\label{coeffasymp}
We have
\begin{align*}
	\int_{-a}^{b} \varphi(t) t^n \,dt &= (-1)^n f_1(0) \Gamma(\mu+1) n^{-\mu-1} a^{n+\mu+1} + O\!\left(n^{-\mu-2} a^n\right) \\
			&\qquad + f_2(0) \Gamma(\nu+1) n^{-\nu-1} b^{n+\nu+1} + O\!\left(n^{-\nu-2} b^n\right).
\end{align*}
as $n \to \infty$.
\end{lemma}

\begin{proof}
If $a \neq 0$ we calculate
\begin{align*}
\int_{-a}^{0} \varphi(t) t^n \,dt &= (-a)^n \int_{-a}^{0} \varphi(t) e^{n \log(-t/a)} \,dt \\
			&= (-a)^n \int_0^a \varphi(s-a) e^{n \log(1-s/a)} \,ds \\
			&= (-a)^n \int_0^a s^{\mu} f_1(s) e^{n \log(1-s/a)} \,ds.
\end{align*}
Letting $s = a(1-e^{-r})$ gives
\begin{align*}
\int_{-a}^{0} \varphi(t) t^n \,dt &= (-1)^n a^{n+\mu+1} \int_{0}^{\infty} (1-e^{-r})^{\mu} f_1(a-ae^{-r}) e^{-r} e^{-nr} \,dr \\
			&= (-1)^n a^{n+\mu+1} \int_{0}^{\infty} r^{\mu} \psi_a(r) e^{-nr} \,dr,
\end{align*}
where
\[
	\psi_a(r) = \left(\frac{1-e^{-r}}{r}\right)^{\mu} f_1(a-ae^{-r}) e^{-r}
\]
has a bounded derivative in a neighborhood of $r=0$.  We may now apply Watson's lemma to conclude that
\begin{align*}
\int_{-a}^{0} \varphi(t) t^n \,dt &= (-1)^n \psi_a(0) \Gamma(\mu+1) n^{-\mu-1} a^{n+\mu+1} + O\!\left(n^{-\mu-2} a^n\right) \\
			&= (-1)^n f_1(0) \Gamma(\mu+1) n^{-\mu-1} a^{n+\mu+1} + O\!\left(n^{-\mu-2} a^n\right).
\end{align*}

Using an identical argument we find that
\[
	\int_0^b \varphi(t) t^n \,dt = f_2(0) \Gamma(\nu+1) n^{-\nu-1} b^{n+\nu+1} + O\!\left(n^{-\nu-2} b^n\right),
\]
which completes the proof.
\end{proof}

A similar argument can be used to prove the following.

\begin{lemma}
\label{tailasymp}
\begin{align*}
\int_{-a}^{b} \frac{\varphi(t)}{1-zt} \,t^{n+1} \,dt &= (-1)^{n+1} \frac{f_1(0) \Gamma(\mu+1)}{1+az} \,n^{-\mu-1} a^{n+\mu+2} + O\!\left(n^{-\mu-2} a^n\right) \\
			&\qquad + \frac{f_2(0) \Gamma(\nu+1)}{1-bz} \,n^{-\nu-1} b^{n+\nu+2} + O\!\left(n^{-\nu-2} b^n\right)
\end{align*}
as $n \to \infty$ uniformly when $z$ is restricted to a compact subset of the doubly-slit plane \linebreak $\C \cap (-\infty,-1/a]^c \cap [1/b, +\infty)^c$.
\end{lemma}

We may now prove the first theorem in this section.

\begin{proof}[Proof of \texorpdfstring{\hyperref[curvetheo1]{Theorem \ref*{curvetheo1}}}{Theorem \ref*{curvetheo1}}]
By definition we have
\[
	F(nz) = \int_{-a}^{b} \varphi(t) e^{nzt} \,dt,
\]
and
\[
	s_n[F](nz) = \int_{-a}^{b} \varphi(t) s_n[\exp](nzt) \,dt.
\]
Subtracting these we get
\begin{align}
\label{Fsndiff}
F(nz) - s_n[F](nz) &= \int_{-a}^{b} \varphi(t) \left(e^{nzt} - s_n[\exp](nzt)\right)\,dt \nonumber \\
					  &= \int_{-a}^{b} \varphi(t) e^{nzt} g_n(zt)\,dt,
\end{align}
where
\[
	g_n(z) = 1 - e^{-nz}s_n[\exp](nz).
\]
It was shown by Szeg\H{o} in \cite{szego:exp} (see also \cite{cvw:expasympi}, \cite{boyergoh:euler}, and \cite{norfolk:1f1}) that
\begin{equation}
\label{szegoapprox}
	g_n(z) = \frac{\left(ze^{1-z}\right)^n}{\sqrt{2 \pi n}} \cdot \frac{z}{1-z} \Bigl(1 - \epsilon_n(z)\Bigr),
\end{equation}
where $\epsilon_n(z) = O(1/n)$ as $n \to \infty$ uniformly when $z$ is restricted to a compact subset of the region $U$ defined in \eqref{uregiondef}.  Upon substituting this into equation \eqref{Fsndiff} we get
\begin{equation}
	F(nz) - s_n[F](nz) = \frac{e^n z^{n+1}}{\sqrt{2 \pi n}} \int_{-a}^{b} \frac{\varphi(t)}{1-zt} \,t^{n+1} \Bigl(1 - \epsilon_n(zt)\Bigr) \,dt.
\label{fatherofmaindifferencezero}
\end{equation}
It follows that zeros of $s_n[F](nz)$ which remain in compact subsets of the region $V_{a,b}$ defined in \eqref{vregiondef} satisfy
\begin{equation}
F(nz) = \frac{e^n z^{n+1}}{\sqrt{2 \pi n}} \int_{-a}^{b} \frac{\varphi(t)}{1-zt} \,t^{n+1} \,dt\, \Bigl(1 + O(1/n)\Bigr)
\label{maindifferencezero}
\end{equation}
as $n \to \infty$, where we have used \hyperref[tailasymp]{Lemma \ref*{tailasymp}} to bring the error term outside of the integral.

Suppose first that $\{z_n\}$ is a sequence in $\C$ such that $s_n[F](nz_n) = 0$ for all $n$, and such that the sequence has a limit point in $V_{a,b} \cap \{z \in \C : \Re(z) < 0\}$.  This implies there is a $\delta > 0$ such that $|z_n + 1/a| > \delta$ for $n$ large enough.

It follows from \hyperref[Fasymp]{Lemma \ref*{Fasymp}} that
\begin{equation}
	|F(nz_n)|^{1/n} = |e^{-az_n}| \left(1 - \left(\Re(\mu) + 1\right) \frac{\log n}{n} + O(1/n)\right)
\label{leftasymp1}
\end{equation}
as $n \to \infty$, and if $c = \max\{a,b\}$ and
\[
	\xi = \begin{cases}
			\Re(\mu) & \text{if } a > b, \\
			\Re(\nu) & \text{if } a < b, \\
			\min\{\Re(\mu),\Re(\nu)\} & \text{if } a = b,
		  \end{cases}
\]
then we have from \hyperref[tailasymp]{Lemma \ref*{tailasymp}} that
\begin{equation}
	\left|\frac{e^n z_n^{n+1}}{\sqrt{2\pi n}} \int_{-a}^{b} \frac{\varphi(t)}{1-z_nt} \,t^{n+1} \,dt\right|^{1/n} = |ecz_n| \left(1 - \left(\xi + \frac{3}{2}\right) \frac{\log n}{n} + O(1/n)\right)
\label{rightasymp}
\end{equation}
as $n \to \infty$.  Upon substituting equations \eqref{leftasymp1} and \eqref{rightasymp} into equation \eqref{maindifferencezero} we see that these zeros $z_n$ satisfy
\[
	\left|c z_n e^{1+az_n}\right| = 1 + \left(\xi - \Re(\mu) + \frac{1}{2}\right) \frac{\log n}{n} + O(1/n)
\]
as $n \to \infty$, which proves part (i) of \hyperref[curvetheo1]{Theorem \ref*{curvetheo1}}.

Suppose now that $\{z_n\}$ is a sequence such that $s_n[F](nz_n) = 0$ for all $n$ and such that the sequence has a limit point in $V_{a,b} \cap \{z \in \C : \Re(z) > 0\}$.  This implies there is a $\delta > 0$ such that $|z_n - 1/b| > \delta$ for $n$ large enough.

Here it follows from \hyperref[Fasymp]{Lemma \ref*{Fasymp}} that
\[
	|F(nz_n)|^{1/n} = |e^{bz_n}| \left(1 - \left(\Re(\nu) + 1\right) \frac{\log n}{n} + O(1/n)\right)
\]
as $n \to \infty$.  Substituting this and equation \eqref{rightasymp} into equation \eqref{maindifferencezero} we see that these zeros $z_n$ satisfy
\[
	\left|c z_n e^{1-bz_n}\right| = 1 + \left(\xi - \Re(\nu) + \frac{1}{2}\right) \frac{\log n}{n} + O(1/n)
\]
as $n \to \infty$, which proves part (ii) of \hyperref[curvetheo1]{Theorem \ref*{curvetheo1}}.
\end{proof}

In the next lemma we will use the result of \hyperref[coeffasymp]{Lemma \ref*{coeffasymp}} to describe where the limit points of the zeros of the normalized sections $s_N[F](Nz)$ may lie.

\begin{lemma}
\label{restrictlemma}
Let $\{N\}$ be a sequence of the natural numbers as described in \hyperref[rosentheo]{Theorem \ref*{rosentheo}}.  If $Z$ is the set of limit points of the zeros of the normalized sections $s_N[F](Nz)$ then $$Z \subseteq \{z \in \C : |z| \leq 2/c\}$$ and $Z$ has no accumulation points in the annulus $$1/c < |z| \leq 2/c.$$
\end{lemma}

\begin{proof}
From Stirling's formula
\[
	k! \sim \left(\frac{k}{e}\right)^k \sqrt{2\pi k}
\]
we have
\[
	(k!)^{1/k} \sim \frac{k}{e}
\]
as $k \to \infty$, and with the aid of \hyperref[coeffasymp]{Lemma \ref*{coeffasymp}} we calculate
\[
	\left|\int_{-a}^{b} \varphi(t) t^k \,dt\right|^{-1/k} \longrightarrow \frac{1}{c}
\]
as $k \to \infty$.  Combining these we see that the power series coefficients of $F(z)$, which are given by
\[
	a_k = \frac{1}{n!} \int_{-a}^{b} \varphi(t) t^k \,dt,
\]
satisfy
\[
	\rho_k = |a_k|^{-1/k} \sim \frac{k}{ec},
\]
where $\rho_k$ is as defined in \hyperref[rosentheo]{Theorem \ref*{rosentheo}}.  Thus the order $\rho$ of $F$ is calculated to be
\[
	\rho = \limsup_{k \to \infty} \frac{\log k}{\log \rho_k} = 1
\]
(see, e.g., \cite[p.~9]{boas:entirefunctions}).

By \hyperref[rosentheo]{Theorem \ref*{rosentheo}}, all zeros of the section $s_N[f](z)$ lie in
\[
	|z| \leq (2+\epsilon)\frac{N}{c}
\]
for $N$ large enough, which tells us that all points of $Z$ must lie in $|z| \leq 2/c$.  The theorem also tells us that the number of zeros of $s_N[F](z)$ with
\[
	|z| > (1+\epsilon)\frac{N}{c}
\]
is bounded for every $\epsilon > 0$.  If $Z$ were to have an accumulation point in the annulus $1/c < |z| \leq 2/c$ then $s_N[f](Nz)$ would necessarily have an unbounded number of zeros with $|z| > (1+\epsilon)/c$ for some $\epsilon > 0$, contradicting the last statement.
\end{proof}

We will now use this result to prove the second theorem in this section.

\begin{proof}[Proof of \texorpdfstring{\hyperref[curvetheo2]{Theorem \ref*{curvetheo2}}}{Theorem \ref*{curvetheo2}}]
Let $Z$ be the set of limit points of the zeros of the normalized sections $s_N[F](Nz)$.  It follows from \hyperref[curvetheo1]{Theorem \ref*{curvetheo1}} that all points of $Z$ in the region $V_{a,b}$ with nonzero real part must lie on the curve $D_{a,b}$ as defined in part (i) of \hyperref[curvetheo2]{Theorem \ref*{curvetheo2}}.  Consequently
\[
	\Bigl(Z \cap V_{a,b}\Bigr) \cup \{\pm 1/c\} \subseteq D_{a,b} \cup \{z \in \C : \Re(z) = 0\} \cup \{\pm 1/c\},
\]
so that
\begin{align}
Z \setminus \Bigl(D_{a,b} \cup \{z \in \C : \Re(z) = 0\} \cup \{\pm 1/c\}\Bigr) &\subseteq Z \setminus \Bigl(V_{a,b} \cup \{\pm 1/c\}\Bigr) \label{zsubset} \\
&\subseteq Z \setminus \{z \in \C : |z| \leq 1/c\} \nonumber \\
&\subseteq \{z \in \C : |z| > 1/c\}. \nonumber
\end{align}
In light of \hyperref[restrictlemma]{Lemma \ref*{restrictlemma}} this shows that the points of $Z$ with nonzero real part not in the set $D_{a,b} \cup \{\pm 1/c\}$ are isolated.

\hyperref[rosentheo]{Theorem \ref*{rosentheo}} assures us that the sequence of sections $\{s_N[F](z)\}$ has a positive fraction of zeros in any sector with vertex at the origin.  It is straightforward to show that for any $0 \leq \theta < 2\pi$ there is a unique $r > 0$ such that $re^{i\theta} \in D_{a,b}$, and since every other point of $Z \setminus \{\pm 1/c\}$ with nonzero real part is isolated it must be true that every point of the curve $D_{a,b}$ is a limit point of zeros.  This proves part (i) of \hyperref[curvetheo2]{Theorem \ref*{curvetheo2}}.

If $\{z_N\}$ is a sequence of complex numbers such that $s_N[F](N z_N) = 0$ for all $N$ which has a limit point on $\{z \in \C : \Re(z) = 0 \,\,\,\text{and}\,\,\, \Im(z) > 1/(ec)\}$ then by equation \eqref{maindifferencezero} and \hyperref[tailasymp]{Lemma \ref*{tailasymp}} we must have $F(Nz_N) \to \infty$.  But if $y \in \R$ then
\[
	|F(iNy)| = \left|\int_{-a}^{b} \varphi(t) e^{iNyt} \,dt\right| \leq \int_{-a}^{b} |\varphi(t)| \,dt,
\]
so such a sequence of zeros cannot exist.  Hence any limit points on the imaginary axis must satisfy $\Im(z) \leq 1/(ec)$.

Further, if $\{z_N\}$ is a sequence of zeros which has a limit point in $|z| \leq 1/(ec)$, then by equation \eqref{fatherofmaindifferencezero} and \hyperref[tailasymp]{Lemma \ref*{tailasymp}} we must have $F(N z_N) \to 0$.  In other words, the zeros $\{z_N\}$ must approximate the zeros of $F(Nz)$.  Conversely, if $\{w_N\}$ is a sequence such that $F(N w_N) = 0$ and $w_n \to iy$ with $|y| \leq 1/(ec)$ then we must likewise have $s_N[F](N w_N) \to 0$, so that the zeros of $F(Nz)$ must approximate the zeros of $s_N[F](Nz)$.  We observe from the asymptotic expansion for $F(nz)$ in \hyperref[Fasymp]{Lemma \ref*{Fasymp}} that the limit points of the zeros of $F(nz)$ all lie on the imaginary axis, and so arrive at the conclusion in part (ii) of \hyperref[curvetheo2]{Theorem \ref*{curvetheo2}}.

The truth of part (iii) follows from combining the facts that a) all points of $Z$ with nonzero real part not in $D_{a,b} \cup \{\pm 1/c\}$ are isolated and lie in $\C \setminus V_{a,b}$, as shown in \eqref{zsubset}; b) the points of $Z$ on the imaginary axis lie in $D_\textnormal{imag}$; and c) all points of $Z$ lie in the disk $|z| \leq 2/c$, as shown in \hyperref[restrictlemma]{Lemma \ref*{restrictlemma}}.  This completes the proof of \hyperref[curvetheo2]{Theorem \ref*{curvetheo2}}.
\end{proof}

\ifpdf
    \graphicspath{{sec_specialcases/PNG/}{sec_specialcases/PDF/}}
\else
    \graphicspath{sec_specialcases/EPS/}
\fi

\section{Special cases of the exponential integrals}
\label{sec_specialcases}

We mentioned in the introduction that the confluent hypergeometric functions with $b > 1$ studied by Norfolk \cite{norfolk:1f1} are a special case of the exponential integrals studied in this paper.  Indeed, when $\Re(b) > 1$ we have the integral representation
\[
	{}_1F_1(1;b;z) = (b-1) \int_0^1 (1-t)^{b-2} e^{zt} \,dt.
\]
Our result extends some of Norfolk's results to the case of complex $b$ with $\Re(b) > 1$.

The prolate spheroidal wave functions are also special cases of these exponential integrals, a consequence of the fact that they satisfy the integral equation
\[
	2i^n R_{0n}^{(1)}(c,1) S_{0n}(c,z) = \int_{-1}^{1} S_{0n}(c,t) e^{iczt}\,dt
\]
(see \cite[sec.~V]{prolatespheroidal}).

In this section we will focus specifically on the class of Bessel functions of the first kind, defined by
\[
	J_{\alpha}(z) = \left(\frac{z}{2}\right)^{\alpha} \sum_{k=0}^{\infty} \frac{(-1)^k}{\Gamma(k+1)\Gamma(k+\alpha+1)} \left(\frac{z}{2}\right)^{2k}.
\]
For $\Re(\alpha) > -1/2$ we have Poisson's integral representation (see, e.g., \cite{watson:bessel})
\[
	J_{\alpha}(z) = \frac{\left(\frac{z}{2}\right)^{\alpha}}{\Gamma\!\left(\alpha + \frac{1}{2}\right) \Gamma\!\left(\frac{1}{2}\right)} \int_{-1}^{1} \left(1-t^2\right)^{\alpha - 1/2} e^{izt} \,dt,
\]
which, once the factor of $z^\alpha$ has been removed and $z$ has been replaced with $-iz$, shows that these Bessel functions are also examples of the exponential integrals.  In the notation of \hyperref[sec_prelims]{Section \ref*{sec_prelims}} we have $a=b=1$, $\mu = \nu = \alpha - 1/2$, and
$$
f_1(t) = f_2(t) = (2-t)^{\alpha-1/2}.
$$
As far as we can tell, asymptotics for the zeros of sections of the Bessel functions have not been previously studied.  We state the result formally as a series of corollaries to \hyperref[curvetheo1]{Theorem \ref*{curvetheo1}} and \hyperref[curvetheo2]{Theorem \ref*{curvetheo2}}.  These results are illustrated in \hyperref[besselzeros]{Figure \ref*{besselzeros}}.

\begin{corollary}
Let $J_\alpha$ be the Bessel function of order $\alpha$ with $\Re(\alpha) > -1/2$, and for $n$ even define
\[
\label{besselsections}
	s_n[J_{\alpha}](z) = \frac{1}{2^\alpha} \sum_{k=0}^{n/2} \frac{(-1)^k}{\Gamma(k+1)\Gamma(k+\alpha+1)} \left(\frac{z}{2}\right)^{2k}
\]
to be its $n^{\text{th}}$ section.  Zeros of $s_n[J_{\alpha}](nz)$ which converge to a point in the region $iV_{1,1} \cap \{z \in \C : \Im(z) > 0\}$ satisfy
\[
	\left|z e^{1+iz}\right| = 1 + \frac{\log n}{2n} + O(1/n)
\]
as $n \to \infty$, and zeros which converge to a point in the region $iV_{1,1} \cap \{z \in \C : \Im(z) < 0\}$ satisfy
\[
	\left|z e^{1-iz}\right| = 1 + \frac{\log n}{2n} + O(1/n)
\]
as $n \to \infty$.
\end{corollary}

\begin{proof}
This is a direct application of \hyperref[curvetheo1]{Theorem \ref*{curvetheo1}}.  We should note that the quantity in \eqref{ncond} is just $(-1)^n + 1$ in this case, and since $n$ is even this is indeed bounded away from zero.
\end{proof}

\begin{corollary}
Let $\{N\}$ be a subsequence of the even indices $\{n\}$ as described in \hyperref[rosentheo]{Theorem \ref*{rosentheo}}.  Every point of the set
\begin{align*}
	D(J) &= \left\{z \in \C : \Im(z) \geq 0,\,\,\, |z| \leq 1, \,\,\,\text{and}\,\,\, \left|ze^{1+iz}\right| = 1 \right\} \\
			&\qquad \cup \left\{z \in \C : \Im(z) \leq 0,\,\,\, |z| \leq 1, \,\,\,\text{and}\,\,\, \left|ze^{1-iz}\right| = 1 \right\} \\
			&\qquad \cup \left\{x \in \R : -1/e \leq x \leq 1/e \right\}.
\end{align*}
is a limit point of the zeros of the normalized sections $s_N[J_{\alpha}](Nz)$.  Every limit point which is not on $D(J)$ is isolated and must lie in the set
\[
	\{z \in \C : |z| \leq 2\} \setminus \Bigl(iV_{1,1} \cup \{\pm i\}\Bigr),
\]
where $V_{1,1}$ is as defined in \eqref{vregiondef}.
\label{besselcorollary}
\end{corollary}

\begin{figure}[h!tb]
	\centering
	\includegraphics[width=0.45\textwidth]{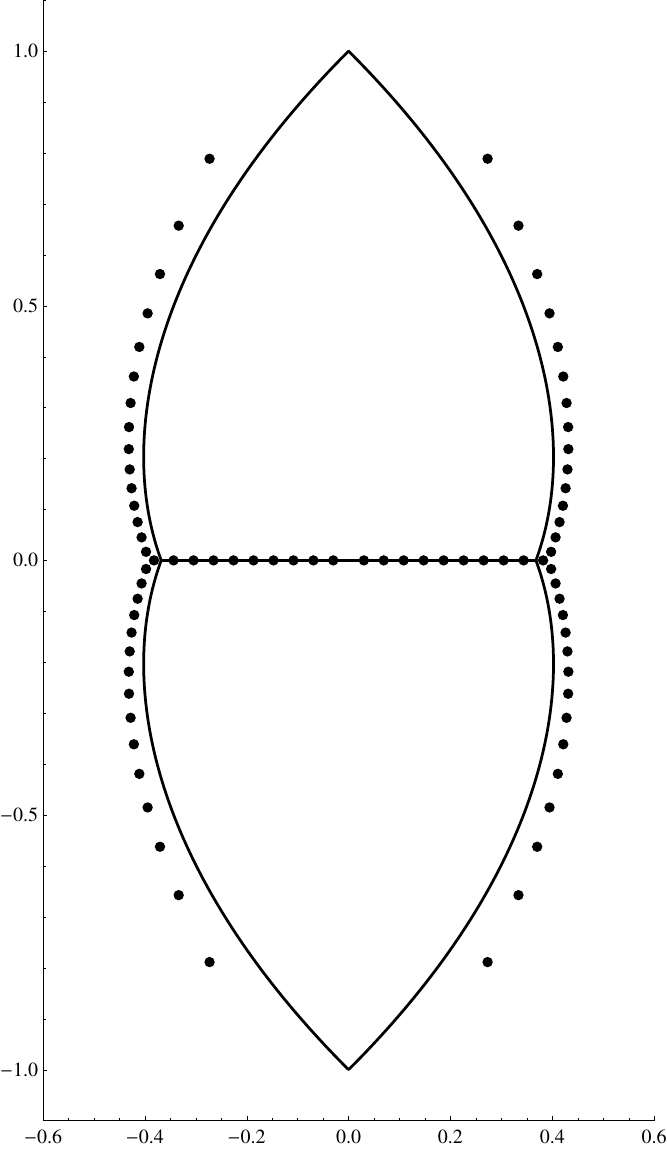}
	\caption{LEFT: Zeros of the normalized section $s_{80}[J_0](80z)$ and the limit curve $D(J)$ as defined in \hyperref[besselcorollary]{Corollary \ref*{besselcorollary}}.}
\label{besselzeros}
\end{figure}

\begin{proof}
The only statement in \hyperref[besselcorollary]{Corollary \ref*{besselcorollary}} which does not follow immediately from \hyperref[curvetheo2]{Theorem \ref*{curvetheo2}} is the assertion that every point of the line segment
$$
S = \left\{x \in \R : -1/e \leq x \leq 1/e \right\}
$$
is a limit point of the zeros of the normalized sections $s_N[J_\alpha](Nz)$.

Near the end of the proof of \hyperref[curvetheo2]{Theorem \ref*{curvetheo2}} we showed that the zeros of $s_N[J_\alpha](Nz)$ approximate the zeros of $J_\alpha(Nz)$ and vice versa in the set $|z| \leq 1/e$.  The zeros of $J_\alpha(z)$ are located at the points $z=k\pi + O(1)$, $k \in \Z$ (see \cite{watson:bessel}), so for any $x \in [-1/e,1/e]$ we can find a sequence $\{z_N\}$ such that $J_\alpha(Nz_N) = 0$ and $z_N \to x$.  It therefore follows that we can find a sequence $\{z_N'\}$ such that $s_N[J_\alpha](Nz_N') = 0$ and $z_N' \to x$, so that every point of $S$ is a limit point of the zeros of the normalized sections.
\end{proof}

We can obtain a slightly different version of the above corollary if we require $\alpha$ to be real.

\begin{corollary}
Let $J_\alpha$ be the Bessel function of real order $\alpha$ with $\alpha > -1/2$, and let $D(J)$ and $s_n[J_\alpha](z)$, where $n$ is a positive, even integer, be defined as in \hyperref[besselcorollary]{Corollary \ref*{besselcorollary}}.  Then every limit point of the zeros of the normalized sections $s_n[J_{\alpha}](nz)$ lies on $D(J)$.
\label{besselcorollary2}
\end{corollary}

\begin{proof}
Let us begin by writing
\[
	s_n[J_{\alpha}](-i n z) = \frac{1}{2^\alpha} P_n\!\left(z^2\right),
\]
where
\[
	P_n(z) = \sum_{k=0}^{n/2} \frac{n^{2k}}{4^k \Gamma(k+1)\Gamma(k+\alpha+1)} \,z^k.
\]
By the Enestr\"om-Kakeya theorem (see, e.g., \cite[ch.~7]{marden:geom} and \cite{asv:ke}), all zeros of $P_n(z)$ satisfy
\[
	|z| \leq 1 + \frac{2\alpha}{n},
\]
so that the limit points of the zeros of the polynomials $P_n$, and hence of the sections $s_n[J_{\alpha}](nz)$, lie in the closed unit disk.

This argument replaces the use of \hyperref[rosentheo]{Theorem \ref*{rosentheo}} and hence \hyperref[restrictlemma]{Lemma \ref*{restrictlemma}} in the proof of \hyperref[curvetheo2]{Theorem \ref*{curvetheo2}}.  Indeed, \hyperref[rosentheo]{Theorem \ref*{rosentheo}} is the origin of the restriction of the indices to the subsequences $\{N\}$.

In particular, it was shown in the proofs of \hyperref[curvetheo1]{Theorem \ref*{curvetheo1}} and \hyperref[curvetheo2]{Theorem \ref*{curvetheo2}} that all limit points in the region $V_{a,b}$ must lie on $D_{a,b} \cup D_\text{imag}$, and this curve is precisely $iD(J)$ in the context of the Bessel functions.  We just showed that all limit points lie in the disk $|z| = 1$, which is a subset of $V_{1,1} \cup \{\pm 1\}$.   It follows that the curve $D(J)$ contains all of the limit points of the zeros of the normalized sections $s_n[J_\alpha](nz)$.
\end{proof}

Note that in comparison with \hyperref[besselcorollary]{Corollary \ref*{besselcorollary}} we no longer have the conclusion that every point of $D(J)$ is a limit point of the zeros of the normalized sections.  This was essentially a consequence of the portion of \hyperref[rosentheo]{Theorem \ref*{rosentheo}} which guaranteed that the sections had a positive fraction of zeros in any sector with vertex at the origin.  In exchange, however, we gain the conclusion that every limit point of the zeros lies on $D(J)$.

Lastly we would like to point out that in defining the sections of the Bessel functions in equation \eqref{besselsections} we have removed their factor of $z^\alpha$.  As such this is a slight abuse of notation; the polynomials in \eqref{besselsections} are technically the sections of the entire functions $z^{-\alpha} J_\alpha(z)$.

\ifpdf
    \graphicspath{{sec_prelims/PNG/}{sec_prelims/PDF/}}
\else
    \graphicspath{sec_prelims/EPS/}
\fi

\section{Discussion}
\label{sec_conclusion}

The exponential integrals in this paper were studied in part because of their simple definitions.  It should be noted that their integral form facilitated calculation and allowed for the use of standard asymptotic techniques.  But they were also studied because their definition was flexible enough to allow the zeros of their sections to display some unique behavior, perhaps the most notable of which is the tendency for the zeros to approach the limit curve from the interior or the exterior depending on the orders of the singularities/zeros of the integrand $\varphi$ at the endpoints of integration (see \hyperref[curvetheoremark]{Remark \ref*{curvetheoremark}}).  

In \hyperref[curvetheo2]{Theorem \ref*{curvetheo2}} we allowed for the possibility that countably many limit points of the zeros do not lie on the limit curve $D_{a,b} \cup D_\text{imag}$.  However, we have been unable to find any examples of these exponential integrals which exhibit pathological behavior; in all of the cases studied numerically the only limit points appear to be the points of $D_{a,b} \cup D_\text{imag}$.  It is possible that there are never any limit points other than those on $D_{a,b} \cup D_\text{imag}$, though we have been unable to prove it.  We state this as a conjecture.

\begin{conjecture}
If $F$ is an exponential integral function as in \hyperref[sec_prelims]{Section \ref*{sec_prelims}} then all of the limit points of the zeros of the normalized sections $s_n[F](nz)$ lie on $D_{a,b} \cup D_\text{imag}$.
\end{conjecture}

It is also possible that, for an exponential integral function $F$, the sequence of sections $\{s_n[F](z)\}$ has a positive fraction of zeros in any sector with vertex at the origin if we allow the indices to run through all of the natural numbers $\N$.  This would allow us to take $\{N\} = \N$ in the statement of \hyperref[curvetheo2]{Theorem \ref*{curvetheo2}}.

\begin{conjecture}
If $F$ is an exponential integral function and $n \in \N$ then the sequence of sections $\{s_n[F](z)\}$ has a positive fraction of zeros in any sector with vertex at the origin.  Consequently, every point of $D_{a,b} \cup D_\text{imag}$ is a limit point of the zeros of the normalized sections $s_n[F](nz)$.
\end{conjecture}

We refer the reader to \cite{acv:angulardistribution}, where the authors study the related problem of the angular distribution of the zeros of the partial sums of the exponential function.


\par \vspace{\baselineskip}

\noindent \small \textbf{Acknowledgments:} The author would like to thank the referees for their valuable questions and comments. \normalsize

\bibliographystyle{amsplain}
\bibliography{biblio}

\end{document}